\documentclass{amsart}

\usepackage{amssymb}
\usepackage{color}
\usepackage{mathtools}
\usepackage[hidelinks]{hyperref} 
\newcommand{\R}{\mathbb{R}}

\newcommand{\N}{\mathbb{N}}
\newcommand{\hd}{\dim_{\textup{H}}}
\newcommand{\bd}{\dim_{\textup{B}}}
\newcommand{\ubd}{\overline{\dim}_{\textup{B}}}
\newcommand{\lbd}{\underline{\dim}_{\textup{B}}}
\newcommand{\uid}{\overline{\dim}_{\,\theta}}
\newcommand{\lid}{\underline{\dim}_{\,\theta}}

\newcommand{\be}{\begin{equation}}
\newcommand{\ee}{\end{equation}}

\DeclareMathOperator*\lowlim{\liminf}
\DeclareMathOperator*\uplim{\limsup}
\DeclarePairedDelimiter{\ceil}{\lceil}{\rceil}
\newtheorem{theorem}{Theorem}[section]
\newtheorem{lemma}[theorem]{Lemma}
\newtheorem{cor}[theorem]{Corollary}
\newtheorem{prop}[theorem]{Proposition}
\theoremstyle{definition}

\theoremstyle{remark}

\newcommand\blfootnote[1]{%
  \begingroup
  \renewcommand\thefootnote{}\footnote{#1}%
  \addtocounter{footnote}{-1}%
  \endgroup
}

\numberwithin{equation}{section}
\begin{document}
\title[Projection Theorems for Intermediate Dimensions]{Projection Theorems for Intermediate Dimensions}
\author{S. A. Burrell}
\address{S. A. Burrell, School of Mathematics and Statistics, University of St Andrews, St Andrews, KY16 9SS, United Kingdom.}
\email{sb235@st-andrews.ac.uk}
\author{K. J. Falconer}
\address{K. J. Falconer, School of Mathematics and Statistics, University of St Andrews, St Andrews, KY16 9SS, United Kingdom.}
\email{kjf@st-andrews.ac.uk}
\author{J. M. Fraser}
\address{J. M. Fraser, School of Mathematics and Statistics, University of St Andrews, St Andrews, KY16 9SS, United Kingdom.}
\email{jmf32@st-andrews.ac.uk}

\date{\today}
\dedicatory{To the memory of John Marstrand 1926-2019}

\begin{abstract}
Intermediate dimensions were recently introduced to interpolate between the Hausdorff and box-counting  dimensions of fractals. Firstly, we show that these intermediate dimensions may be defined in terms of capacities with respect to certain kernels. Then, relying on this, we show that the intermediate dimensions of the projection of a set $E \subset \R^n$ onto almost all $m$-dimensional subspaces depend  only on $m$ and $E$, that is, they are almost surely independent of the choice of subspace. Our approach is based on   `intermediate dimension profiles' that are expressed in terms of capacities.  We discuss several applications at the end of the paper, including a surprising result that relates the  \emph{box} dimensions of the projections of a set to the \emph{Hausdorff} dimension of the  set.
\end{abstract}

\maketitle
\blfootnote{First published in: S. A. Burrell, K. J. Falconer, J. M. Fraser, Projection theorems for intermediate dimensions. \\
\hphantom{0in}J. Fractal Geom. 8 (2021), 95-116. doi: 10.4171/JFG/99. \textsuperscript{\textcopyright} European Mathematical Society.}
\section{Introduction}
Theorems on dimensions of projections of fractals in Euclidean space have a long history. In 1954 Marstrand \cite{mar:1954} proved that the Hausdorff dimension of the orthogonal projections of a Borel set $E \subset \R^2$ onto linear subspaces was almost-surely constant. More specifically,
\begin{equation*}
\hd \pi_V E = \min\{\hd E, 1\},
\end{equation*}
for almost all one-dimensional subspaces $V$, where $\pi_V$ denotes orthogonal projection onto $V$. Kaufman gave a potential-theoretic proof of Marstrand's results \cite{kau:1968}, and in 1975 Mattila extended them to Borel sets $E  \subset \R^n$ and almost all $V$ in the Grassmannian $G(n, m)$ \cite{mat:1975}.  These seminal results set in motion a sustained interest  in the behaviour of dimension under projections, see \cite{Falconer,mat:book} for basic expositions and \cite{fafaji:2015,mat:2014,shmer:2015} for recent surveys.

It is natural to seek projection results for the various other dimensions that occur throughout fractal geometry. For example, in 1997 Falconer and Howroyd showed that the upper and lower box-counting dimensions of the projections of a set are almost surely constant and given by what they termed a `dimension profile' \cite{faho:1997,ho:2001}, reflecting how a set in $\R^n$ appears when viewed from an $m$-dimensional perspective. The dimension profiles were, however, implicitly defined and somewhat awkward to work with, leading to a recent re-working of the theory using a potential-theoretic approach \cite{fal:2018, fal:2019} where box-counting  dimensions are defined in terms of capacities, which are then used to study projections. 

Recently,  Falconer, Fraser and Kempton \cite{fafrke:2018} introduced intermediate dimensions to provide a continuum of dimensions, one for each $\theta \in [0, 1]$, that interpolate between the  Hausdorff dimension (obtained when $\theta = 0$) and box-counting dimensions ($\theta = 1$). These dimensions are defined by restricting the diameters of sets used in admissible coverings of $E$ to a range $[r,r^\theta]$ for small $r$.  A general discussion of this and other forms of dimension interpolation may be found in the recent survey \cite{fra:survey:2019}. 

In this paper, potential-theoretic methods are used to study intermediate dimensions, first to give a definition of these dimensions in terms of capacities with respect to certain kernels and then to prove a Marstrand-type theorem to give the almost sure intermediate dimensions of projections of sets in terms of capacities, see Theorem \ref{main}. Some examples and applications are given in the final section.

\section{Intermediate dimensions}\label{setting}
Intermediate dimensions were introduced by Falconer, Fraser and Kempton in \cite{fafrke:2018} to interpolate between Hausdorff dimension and box-counting dimensions. The lower and upper intermediate dimensions, $\lid E$ and  $\uid E$, of a set $E\subset\R^n$ depend on a parameter $\theta\in [0,1]$, with $\underline{\dim}_0 E = \overline{\dim}_0 E = \hd E$ and $\underline{\dim}_1 E = \lbd E$ and $\overline{\dim}_1 E = \ubd E$ where $\hd, \lbd$ and $\ubd$ denote Hausdorff, lower box- and upper box-counting dimension, respectively. Various properties of intermediate dimensions are established in \cite{fafrke:2018}. In particular $\lid E$ and $\uid E$ are monotonically increasing in $\theta\in [0,1]$, are continuous except perhaps at $\theta = 0$, and are invariant under bi-Lipschitz mappings. Intermediate dimensions are of interest for sets which have differing Hausdorff and box-counting dimensions, such as sequence sets of the form $\{0\} \cup  \{n^{-p} : n=1, 2, \dots\}$ for $p>0$, self-affine carpets and many other examples, with the intermediate dimensions reflecting the range of diameters of sets needed to get coverings that are efficient for estimating dimensions, see \cite{fafrke:2018}.

Specifically, for $E \subset \R^n$ and $0 < \theta \leq 1$, the  {\em lower intermediate dimension} of $E$ may be defined as
\begin{align}\label{lid}
\lid E =  \inf \big\{& s\geq 0  :  \mbox{ \rm for all $\epsilon >0$ and all $r_0>0$, there exists }  \nonumber\\
&\mbox{ $0<r\leq r_0$ and a cover $ \{U_i\} $ of $E$  such that} \\
 & \mbox{ $r^{1/\theta} \leq  |U_i| \leq r $ and 
 $\sum |U_i|^s \leq \epsilon$}  \big\}\nonumber
\end{align}
and the corresponding {\em upper intermediate dimension} by
\begin{align}\label{uid}
\uid E =  \inf \big\{& s\geq 0  :  \mbox{ \rm for all $\epsilon >0$, there exists $r_0>0$ such that} \nonumber\\
& \mbox{for all $0<r\leq r_0$, there is a cover $ \{U_i\} $ of $E$} \\
&\mbox{such that $r^{1/\theta} \leq  |U_i| \leq r$ and 
$\sum |U_i|^s \leq \epsilon$}  \big\},\nonumber
\end{align}
where $|U|$ denotes the diameter of a set $U \subset \R^n$. When $\theta = 0$ we take \eqref{lid} and \eqref{uid}
with no lower bounds  on the diameters of covering sets, recovering the Hausdorff dimension in both cases. When $\theta=1$ all covering sets are forced to have the same diameter and we recover the lower and upper box-counting dimensions, respectively.

For our purposes it is convenient to work with equivalent definitions of these intermediate dimensions in terms of limits of logarithms of sums over covers.
For bounded and non-empty $E \subset \R^n$, $\theta  \in (0, 1]$ and $s\in [0,n]$, define
\begin{align}\label{sums}
S_{r, \theta}^s(E) := \inf \Big\{& \sum_i |U_i|^s :\mbox{ \rm $\{U_i\}_i$\textnormal{ is a cover of} $E$ \textnormal{ such that }}\nonumber\\ 
& \mbox{ $r \leq |U_i| \leq r^\theta$\,\,\textnormal{ for all } $i$} \Big\}.
\end{align}
We claim
\be\label{altlid}
\lid E =  \bigg(\textnormal{ the unique } s\in [0,n] \textnormal{ such that  } \liminf\limits_{r \rightarrow 0} \frac{\log S_{r, \theta}^s(E)}{-\log r} =0\bigg)
\ee
and 
\be\label{altuid}
\uid E =  \bigg(\textnormal{ the unique } s\in [0,n] \textnormal{ such that  } \limsup\limits_{r \rightarrow 0} \frac{\log S_{r, \theta}^s(E)}{-\log r} =0\bigg).
\ee
 It is easy to see from \eqref{lid} and \eqref{uid} that  $\lid E$ and $\uid E$ are the infima of $s$ for which these lower and upper limits equal 0; that there are unique such values follows from the following lemma.

\begin{lemma}\label{declem}
Let $\theta \in (0, 1]$ and $E \subset \R^n$. For each $0<r<1$,
\be\label{lip}
-(s - t) \leq \frac{\log S_{r, \theta}^s(E)}{-\log r} - \frac{\log S_{r, \theta}^t(E)}{-\log r} \leq -\theta(s- t)
\qquad (0\leq t \leq s \leq n).
\ee
Moreover, there is a unique $s\in [0,n]$ such that $\liminf\limits_{r \rightarrow 0} \frac{\log S_{r, \theta}^s(E)}{-\log r}=0$ and a unique $s\in [0,n]$ such that $\limsup\limits_{r \rightarrow 0} \frac{\log S_{r, \theta}^s(E)}{-\log r}=0.$
\end{lemma}

\begin{proof}
For a cover $\{U_i\}$ of $E$ satisfying $r \leq |U_i| \leq r^\theta$ and $0\leq t \leq s \leq n$, 
$$
\sum_i |U_i|^t r^{s - t} \leq \sum_i |U_i|^s \leq \sum_i |U_i|^t r^{\theta(s - t)}.
$$
Taking infima over all such covers yields
$$
r^{s - t}S_{r, \theta}^t(E)  \leq  S_{r, \theta}^s(E) \leq r^{\theta(s - t)}S_{r, \theta}^t(E),
$$
from which \eqref{lip} follows. These inequalities carry over on taking lower limits of the quotients so in particular $$\liminf\limits_{r \rightarrow 0} \frac{\log S_{r, \theta}^s(E)}{-\log r}$$ is strictly monotonic decreasing and continuous for $s\in [0,n]$. 
Since $S_{r, \theta}^0(E)$ is bounded below by the box-counting number of $E$ at scale $r^\theta$, it follows that 
\[
\liminf\limits_{r \rightarrow 0} \frac{\log S_{r, \theta}^0(E)}{-\log r} \geq \theta \, \lbd E\geq 0.
\]
 Also $S_{r, \theta}^n(E)$ is bounded above by the $n$-dimensional volume of a ball containing $E$ so
\[
\liminf\limits_{r \rightarrow 0} \frac{\log S_{r, \theta}^n(E)}{-\log r}\leq 0. 
\]
 Continuity now gives  a unique $s\in [0,n]$ such that $\liminf\limits_{r \rightarrow 0} \frac{\log S_{r, \theta}^s(E)}{-\log r}=0$. A similar argument holds for upper limits.
\end{proof}

In Section \ref{secapdim} we will show how $\lid E$ and $\uid E$ can be represented in terms of capacities of $E\subset\R^n$ with respect to certain kernels. Then in Section \ref{projsec} we will show that by changing a parameter in the kernels we obtain the intermediate dimensions of the orthogonal projections of $E$ onto almost all $m$-dimensional subspaces. 

\section{Capacities and Dimension Profiles}
In this section we introduce a notion of dimension derived from capacities that is closely related to the intermediate dimensions and which is amenable to studying projections. 

Throughout this section, let $\theta \in (0, 1]$ and $m \in \{1, \dots, n\}$. For $0 \leq s \leq m$ and $0<r <1$, define the potential kernels
\be\label{ker}
{\phi}_{r, \theta}^{s, m}(x) = \begin{cases} 
      1 & 0\leq |x| < r \\
      \big(\frac{r}{|x|}\big)^s & r\leq |x| < r^\theta   \\
      \frac{r^{\theta(m-s) + s}}{|x|^m}\ & r^\theta \leq |x|
   \end{cases} \qquad (x\in \R^n).
\ee
When $s = m$ this becomes
\be\label{kerb}
\phi_{r, \theta}^{m, m}(x) = \begin{cases} 
      1 &0\leq |x| < r \\
      \big(\frac{r}{|x|}\big)^m & r \leq |x| 
   \end{cases}\qquad (x\in \R^n),
\ee and so corresponds to the kernel $\phi_r^m(x)$ used in \cite{fal:2018,fal:2019} in the context of box-counting  dimensions. As one would expect, this kernel is also recovered when $\theta =1$ where ${\phi}_{r, \theta}^{s, m}$ is independent of $s$. Note that  ${\phi}_{r, \theta}^{s, m}(x)$ is continuous in $x$ and monotonically decreasing in $|x|$.     Letting $\mathcal{M}(E)$ denote the set of Borel probability measures supported on $E$, we say that the \emph{energy} of $\mu \in \mathcal{M}(E)$ with respect to $\phi_{r, \theta}^{s, m}$ is  
\begin{equation*}
\int\int \phi_{r, \theta}^{s, m}(x - y) \,d\mu(x)d\mu(y)
\end{equation*}
and the {\em potential} of $\mu$ at $x \in \R^n$ is
\begin{equation*}
\int \phi_{r, \theta}^{s, m}(x - y)\,d\mu(y).
\end{equation*}
We define the \emph{capacity} $C_{r, \theta}^{s, m}(E)$ of $E$ to be the reciprocal of the minimum energy achieved by probability measures on $E$, that is
\begin{equation*}
{C_{r, \theta}^{s, m}(E)} = \left(\inf\limits_{\mu \in \mathcal{M}(E)} \int\int \phi_{r, \theta}^{s, m}(x - y) \,d\mu(x)d\mu(y)\right)^{-1}.
\end{equation*}
Since $\phi_{r, \theta}^{s, m}(x)$ is continuous in $x$ and strictly positive and $E$ is compact, $C_{r,\theta}^{s, m}(E)$ is positive and finite. For bounded, but not necessarily closed, sets we take the capacity  of a set to be that of its closure.\\

The existence of equilibrium measures for kernels and the relationship between the minimal energy and the corresponding potentials is standard in classical potential theory. We state this in a convenient form; it is easily proved for continuous kernels, see, for example, \cite[Lemma 2.1]{fal:2019}. 

\begin{lemma}\label{attaincap}
Let $E\subset \R^n$ be compact, $m \in \{1,\dots,n\}$, $0 \leq s \leq m$, $\theta \in (0, 1]$  and $0<r<1$. Then there exists an equilibrium measure $\mu \in \mathcal{M}(E)$ such that 
$$\int\int \phi_{r, \theta}^{s, m}(x - y) d\mu(x)d\mu(y) = \frac{1}{C_{r, \theta}^{s, m}(E)} =: \gamma.$$
Moreover, 
$$\int \phi_{r, \theta}^{s, m}(x - y) d\mu(y) \geq \gamma$$
for all $x \in E$, with equality for $\mu$-almost all $x \in E$.
\end{lemma}

As we will see, these capacities are closely related to the sums considered in Section \ref{setting}. The following lemma, which parallels Lemma \ref{declem}, enables us to define `intermediate dimension profiles'.

\begin{lemma}\label{caplem}
Let $E\subset \R^n$ be compact,  $m \in \{1,\dots,n\}$, $\theta \in (0, 1]$ and $E \subset \R^n$. If $0<r<1$, then for all $0\leq t \leq s \leq m$,
\be\label{lip2}
-(s - t) \leq \bigg(\frac{\log C_{r, \theta}^{s, m}(E)}{-\log r} -s\bigg)  - \bigg(\frac{\log C_{r, \theta}^{t, m}(E)}{-\log r}-t\bigg) \leq -\theta(s- t)
.
\ee
Moreover, there is a unique $\underline{s}\in [0,m]$ such that $\liminf\limits_{r \rightarrow 0} \frac{\log C_{r, \theta}^{\underline{s}, m}(E)}{-\log r}=\underline{s}$ and a unique $\overline{s}\in [0,m]$ such that $\limsup\limits_{r \rightarrow 0} \frac{\log C_{r, \theta}^{\overline{s}, m}(E)}{-\log r}=\overline{s}$.
\end{lemma}

\begin{proof}
By comparison of the kernels it is easily checked that, for $0\leq t \leq s \leq m$,
$$\phi_{r, \theta}^{s, m}(x) \leq \phi_{r, \theta}^{t, m}(x)\leq r^{(t-s)(1-\theta)} \phi_{r, \theta}^{s,m}(x)\qquad (x\in \R^n).$$
Using the definition of capacity and that an   equilibrium  measure on $E$ for the kernel $\phi_{r, \theta}^{s, m}$ is a candidate for an equilibrium measure for  $\phi_{r, \theta}^{t, m}$ and vice-versa, we obtain
$$C_{r, \theta}^{s, m}(E)\geq C_{r, \theta}^{t, m}(E)\geq r^{(s-t)(1-\theta)} C_{r, \theta}^{s, m}(E).$$
Taking logarithms  and rearranging gives \eqref{lip2}.

The inequalities \eqref{lip2} remain true on taking lower limits of the quotients so $\liminf\limits_{r \rightarrow 0} \frac{\log C_{r, \theta}^{s, m}(E)}{-\log r}-s $ is strictly monotonic decreasing and continuous in $s\in [0,m]$. With the kernels \eqref{kerb} it is shown in \cite{fal:2019} that  
\[
\liminf\limits_{r \rightarrow 0} \frac{\log C_{r, \theta}^{m, m}(E)}{-\log r}   = \lbd \pi_V E\leq m
\]
 for projections  $\pi_V E$ of $E$ onto almost all $m$-dimensional subspaces $V\in G(n,m)$, so 
\[
\liminf\limits_{r \rightarrow 0} \frac{\log C_{r, \theta}^{m, m}(E)}{-\log r} -m \leq 0.
\]
 Since the kernels are bounded above by 1, $C_{r, \theta}^{0,m}(E)\geq 1$, so 
\[
\liminf\limits_{r \rightarrow 0} \frac{\log C_{r, \theta}^{0, m}(E)}{-\log r}-0 \geq 0.
\]
 We conclude that there is a unique $s\in [0,m]$ such that $\liminf\limits_{r \rightarrow 0} \frac{\log C_{r, \theta}^{s, m}(E)}{-\log r}=s$, and similarly for the upper limits.
\end{proof}

Thus, for each integer $1 \leq m \leq n$, we define the \emph{lower intermediate dimension profile} of $E \subset \R^n$ as
\be\label{lidp}
\lid^m E =  \bigg(\textnormal{ the unique } s\in [0,m] \textnormal{ such that  } \lowlim\limits_{r \rightarrow 0}\frac{\log C_{r, \theta}^{s, m}(E)}{-\log r} = s\bigg)
\ee
and the \emph{upper intermediate dimension profile} as
\be\label{ludp}
\uid^m E =  \bigg(\textnormal{ the unique } s\in [0,m] \textnormal{ such that  } \uplim\limits_{r \rightarrow 0}\frac{\log C_{r, \theta}^{s, m}(E)}{-\log r} = s\bigg).
\ee

\begin{lemma} \label{monoinm}
The intermediate dimension profiles are increasing in $m$, that is, for compact  $E$,  $\theta \in (0,1]$ and $1 \leq m_1 \leq m_2 \leq n$
\[
\lid^{m_1} E \leq \lid^{m_2} E \qquad \text{and} \qquad \uid^{m_1} E \leq \uid^{m_2} E .
\]
\end{lemma}

\begin{proof}
This follows immediately noting that the kernels $ \phi_{r, \theta}^{t, m}(x)$ are clearly \emph{decreasing} in $m$.
\end{proof}

The remainder of the paper concerns the relationship between the intermediate dimensions of a set $E$, defined in terms of the sums over restricted covers of $E$, and intermediate dimension profiles, defined in terms of the capacities. In particular, we will see that for $E \subset \R^n$,   $\lid E = \lid^n E$  and that $\lbd \pi_V E= \lid^m E$ for projections  $\pi_V E$ of $E$ onto almost all $m$-dimensional subspaces $V$ $(1\leq m\leq n-1)$ with respect to the natural invariant measure on the Grassmannian.

\section{Capacities and Intermediate Dimensions}\label{secapdim}

The main result in this section characterises intermediate dimensions of sets $E\subset \R^n$  in terms of dimension profiles which we have defined in terms of capacities $C_{r, \theta}^{s, n}(E)$ with respect to the kernels $\phi_{r, \theta}^{s, n}$. 

\begin{theorem}\label{equivdim}
Let $E \subset \R^n$ be bounded and $\theta \in (0, 1]$. Then
\begin{equation*}
\lid E = \lid^{n} E
\end{equation*}
and
\begin{equation*}
\uid E = \uid^{n} E.
\end{equation*}
\end{theorem}

This will follow immediately from the following proposition together with the definitions \eqref{altlid},  \eqref{altuid}, \eqref{lidp} and \eqref{ludp}. We may assume throughout that $E$ is compact since the intermediate dimensions are stable under taking closures, see \cite{fafrke:2018}.

\begin{prop}\label{approxeq}
Let $E \subset \R^n$ be compact, $\theta \in (0, 1]$, and $0\leq s \leq n$. Then there is a number $r_0>0$ such that for all $0<r\leq r_0$,
\be\label{ineqs}
r^sC_{r, \theta}^{s, n}(E) \leq S_{r, \theta}^{s}(E) \leq a_n\ceil{\log_2( |E|/r) + 1}r^sC_{r, \theta}^{s, n}(E)
\ee
where the number $a_n$ depends only on $n$. Consequently
$$
\liminf\limits_{r \rightarrow 0} \frac{\log S_{r, \theta}^s(E)}{-\log r} = -s + \liminf\limits_{r \rightarrow 0}\frac{\log C_{r, \theta}^{s, n}(E)}{-\log r}
$$
and
$$
\limsup\limits_{r \rightarrow 0} \frac{\log S_{r, \theta}^s(E)}{-\log r} = -s + \limsup\limits_{r \rightarrow 0}\frac{\log C_{r, \theta}^{s, n}(E)}{-\log r}.
$$
\end{prop}

\noindent{\bf Proof of Proposition \ref{approxeq}}\label{capintsec}\ 
We prove the left hand inequality of \eqref{ineqs}  in  Lemma \ref{capacitylb} and the right hand inequality in 
Lemma \ref{capacityub}.

\begin{lemma}\label{capacitylb}
Let $E\subset \R^n$ be compact, $\theta \in (0, 1]$, $0<r <1 $ and $0 \leq s \leq n$. 
Then
\be\label{lefth}
r^{s}C_{r, \theta}^{s, n}(E)\leq S_{r, \theta}^s(E).
\ee

\begin{proof}
By Lemma \ref{attaincap} there exists an equilibrium measure $\mu \in \mathcal{M}(E)$ and a set $E_0$ with $\mu(E_0) = 1$ such that
$$\int \phi_{r, \theta}^{s, n}(x - y) d\mu(y) = \frac{1}{C_{r, \theta}^{s, n}(E)} =: \gamma$$
for all $x \in E_0$. Let $r \leq \delta \leq r^\theta$ and $x \in E_0$. Then
\begin{equation}\label{zz}
\gamma = \int {\phi}_{r, \theta}^{s, n}(x - y) d\mu(y) 
\geq \int \left(\frac{r}{\delta}\right)^s1_{B(0, \delta)}(x - y) d\mu(y) 
\geq \left(\frac{r}{\delta}\right)^s \mu(B(x, \delta)).
\end{equation}
Let $\{U_i\}_{i }$ be a finite cover of $E$ by sets of diameters $r \leq |U_i| \leq r^\theta$ and define $\mathcal{I} = \{i : U_i \cap E_0 \neq \emptyset \}$. Then for each $i \in \mathcal{I}$, there exists $x_i \in U_i \cap E_0$ so that $U_i \subset B(x_i, |U_i|)$. Hence
\begin{equation*}
1 = \mu(E_0) \leq \sum\limits_{i \in \mathcal{I}} \mu(U_i) 
\leq  \sum\limits_{i \in \mathcal{I}} \mu(B(x_i, |U_i|))\leq r^{-s}\gamma \sum\limits_{i \in \mathcal{I}}|U_i|^s
\end{equation*}
by (\ref{zz}), and so
\begin{equation*}
\sum\limits_{i } |U_i|^s \geq r^sC_{r,\theta}^{s, n}(E),
\end{equation*}
which yields the desired result upon taking the infimum over all such covers.
\end{proof}
\end{lemma}

Note that by comparing kernels, $C_{r, \theta}^{s, m}(E) \leq C_{r, \theta}^{s, n}(E)$ for $m\leq n$ so \eqref{lefth} implies the weaker conclusion that  $r^{s}C_{r, \theta}^{s, m}(E)\leq S_{r, \theta}^s(E)$.

In the following proof, we use potential estimates to find a Besicovitch cover of $E$ by balls of relatively large measure. The Besicovitch covering lemma gives a bounded number of families of disjoint such balls with their union covering $E$. The balls with diameters between $r$ and $r^\theta$, together with covers of any larger balls by balls of diameters at most $r^\theta$, provide efficient covers for estimating the sums $S_{r,\theta}^s(E)$. Additionally, in the next section, Lemma \ref{capacityub} will be important when considering intermediate dimensions of projections.

\begin{lemma}\label{capacityub}
Let $E\subset \R^n$ be compact, $0 \leq s \leq n$ and  $\theta \in (0, 1]$. If there exists a measure $\mu \in \mathcal{M}(E)$ and $\gamma > 0$ such that
\begin{equation}\label{cond}
\int \phi_{r, \theta}^{s, n}(x- y) d\mu(y) \geq \gamma
\end{equation}
for all $x \in E$, then there is a number $r_0>0$ such that for all $0<r\leq r_0$,
$$ S_{r,\theta}^s(E) \leq a_n\ceil{\log_2(|E|/r)+1}\frac{r^{s}}{\gamma}$$
where the constant $a_n$ depends only on $n$. In particular, 
$$ S_{r,\theta}^s(E) \leq a_n\ceil{\log_2(|E|/r)+1} C_{r,\theta}^{s, n}(E)r^s.$$

\begin{proof}
To avoid ambiguity we will assume that $\theta \in (0, 1)$, though the proof is virtually the same when $\theta=1$, essentially by taking $M=0$; this `box-counting dimension' case is also covered in \cite{fal:2019}.

Let $D = \ceil{\log_2(|E|/r)}$ and let $M$ be the integer satisfying
\be\label{Mdef}2^{M-1}r < r^\theta \leq 2^{M}r.
\ee
We choose $r_0$ sufficiently small to ensure  that $2\leq M\leq D-2$ for all $0<r\leq r_0$.
For $x \in E$, using \eqref{cond} and estimating the kernel $\phi_{r, \theta}^{s, n}(x - y)$ given by \eqref{ker} over consecutive annuli  $B(x, 2^{k}r)\setminus B(x, 2^{k-1}r)\ (1\leq k\leq D)$, 
\begin{align*}
\gamma &\leq \int \phi_{r, \theta}^{s, n}(x -y)d\mu(y)\\ 
&\leq \mu(B(x, r)) + \sum\limits_{k=1}^{D} \int_{B(x, 2^{k}r)\setminus B(x, 2^{k-1}r)}\phi_{r, \theta}^{s, n}(x - y)d\mu(y)\\
&\leq \mu(B(x, r)) + \sum\limits_{k=1}^{M}\int_{B(x, 2^{k}r)\setminus B(x, 2^{k-1}r)}2^{-(k-1)s}d\mu(y) \\
&\quad + \sum\limits_{k = M+1}^{D} \int_{B(x, 2^{k}r)\setminus B(x, 2^{k-1}r)}r^{\theta(n-s) + s}(2^{k-1}r)^{-n}d\mu(y)\\
&\leq \sum\limits_{k=0}^{M-2}2^{s} \mu(B(x, 2^{k}r))2^{-ks} +  \sum\limits_{k=M-1}^{M}2^{s} \mu(B(x, 2^{k}r))2^{-ks} \\
&\quad +r^{(\theta-1)(n-s)}\sum\limits_{k = M+1}^{D} \mu(B(x, 2^{k}r))
2^{-(k-1)n}.
\end{align*}
Hence, for each $x \in E$, there exists some integer $0 \leq k(x) \leq D$    such that one of the above summands is at least the arithmetic mean of the sum. There are three cases. We will use that there are numbers $d_n$ depending only on $n$ such that every ball of radius $\rho$ in $\R^n$ may be covered by at most $\lambda^{-n} d_n $ balls of diameter $\lambda \rho$ for all $0<\lambda \leq 1$ ($d_n = 3^n n^{n/2}$ will certainly do).
\medskip 

(i) If $0 \leq k(x) \leq M-2$ then 
$$
\frac{\gamma}{D + 1} \leq 2^s\mu(B(x, 2^{k(x)}r))2^{-k(x)s} = 4^s\mu(B(x, 2^{k(x)}r))|B(x, 2^{k(x)}r)|^{-s}r^s,
$$
so  
\begin{equation}\label{bdiam}
|B(x, 2^{k(x)}r)|^{s} \leq (D+1)\gamma^{-1}  4^sr^s\mu(B(x, 2^{k(x)}r));
\end{equation}

(ii) if $M - 1 \leq k(x) \leq M$ then
\begin{align*}
\frac{\gamma}{D + 1} &\leq 2^s\mu(B(x, 2^{k(x)}r))2^{-k(x)s} \\
& \leq \mu(B(x, 2^{k(x)}r))2^{s}2^{-(M-1)s}
\leq \mu(B(x, 2^{k(x)}r))2^{2s}r^{(1-\theta)s},
\end{align*}
so 
\begin{equation}\label{cdiam}
4^{n} d_n\, r^{\theta s} \leq 4^n2^{2s}(D+1)\gamma^{-1}d_nr^{s} \mu(B(x, 2^{k(x)}r));
\end{equation}

(iii) if $M + 1 \leq k(x) \leq D$ then
$$
\frac{\gamma}{D + 1}  \leq r^{(\theta-1)(n-s)}\mu(B(x, 2^{k(x)}r))2^{-(k(x) - 1)n},
$$
so 
\begin{equation}\label{ddiam}
d_n 2^{k(x)n} r^{(1-\theta)n} \leq 2^n(D+1)\gamma^{-1}d_nr^{s(1-\theta)} \mu(B(x, 2^{k(x)}r)).
\end{equation}
The cover of $E$ by the balls $\mathcal{B} = \{B(x, 2^{k(x)}r) : x \in E\}$ is a Besicovitch cover, that is each point of $E$ is at the centre of some ball in the collection. The Besicovitch covering theorem, see for example \cite[Theorem 2.7]{mat:1975}, allows us to extract  subcollections $\mathcal{C}_1, \dots, \mathcal{C}_{c_n}$ of disjoint balls from $\mathcal{B}$ where $c_n$ depends only on $n$ and such that $E \subset \bigcup_i \bigcup\limits_{B \in \mathcal{C}_i} B$.  Let 
$$\mathcal{E}_i = \{B(x, 2^{k(x)}r) \in \mathcal{C}_i: M - 1 \leq k(x) \leq M\}$$
and
$$
\mathcal{F}_i = \{B(x, 2^{k(x)}r) \in \mathcal{C}_i: M + 1 \leq k(x) \leq D\}. 
$$
From \eqref{Mdef} each $B\in \mathcal{C}_i \setminus (\mathcal{E}_i\cup \mathcal{F}_i)$ has diameter at most $r^\theta$.
Also, for each $B = B(x, 2^{k(x)}r) \in \mathcal{E}_i$ let $\mathcal{D}_B$ denote a collection of at most $ (2^M r /r^\theta)^n d_n  \leq 2^n d_n $ balls of diameter $r^\theta$ that cover $B$, and  for each $B = B(x, 2^{k(x)}r) \in \mathcal{F}_i$ let $\mathcal{D}_B$ denote a collection of at most $\big(2^{k(x)}r/r^\theta\big)^n d_n$ balls of diameter $r^\theta$ that cover $B$.\\

For each $i = 1,\dots,c_n$, we consider the cover
$$
\widetilde{\mathcal{C}}_i:= \big(\mathcal{C}_i \setminus (\mathcal{E}_i\cup \mathcal{F}_i)\big) \cup \bigcup\limits_{B \in \mathcal{E}_i \cup\mathcal{F}_i } \mathcal{D}_B
$$
of $\bigcup\limits_{B \in \mathcal{C}_i} B$. Then using \eqref{bdiam} - \eqref{ddiam},
\begin{align*}
\sum\limits_{B \in \mathcal{C}_i \setminus (\mathcal{E}_i\cup \mathcal{F}_i)} |B|^s &+ \sum\limits_{B \in \mathcal{E}_i}\sum\limits_{B' \in \mathcal{D}_B} |B'|^s + \sum\limits_{B \in \mathcal{F}_i}\sum\limits_{B' \in \mathcal{D}_B} |B'|^s\\
&\leq 
4^s(D+1)\frac{r^s}{\gamma}\sum\limits_{B \in \mathcal{C}_i \setminus (\mathcal{E}_i\cup \mathcal{F}_i)} \mu(B) + \sum\limits_{B \in \mathcal{E}_i}4^{n}d_n\,r^{\theta s}\\
&\quad + \sum\limits_{B \in \mathcal{F}_i}d_n\left(\frac{2^{k(x)}r}{r^\theta}\right)^{n}r^{\theta s} \\
&\leq 
4^s(D+1)\frac{r^s}{\gamma} 
+ \sum\limits_{B \in \mathcal{E}_i}\frac{4^n2^{2s}(D+1)d_n}{\gamma}r^{s} \mu(B)\\
&\quad + \sum\limits_{B \in \mathcal{F}_i}\frac{2^n (D+1)d_n}{\gamma}r^{s(1-\theta)}r^{\theta s}\mu(B)\\
&\leq 
4^s(D+1)\frac{r^s}{\gamma} 
+\frac{4^n2^{2s}(D+1)d_n}{\gamma}r^{s}\sum\limits_{B \in \mathcal{E}_i}\mu(B)\\
&\quad + \frac{2^n (D+1)d_n}{\gamma}r^s\sum\limits_{B \in \mathcal{F}_i}\mu(B)\\
&\leq 
(4^{n}+2\cdot4^{2n} d_n)(D+1)\frac{r^s}{\gamma},
\end{align*}
where we have used that  $\mathcal{C}_i$ is a disjoint collection of balls. Hence, writing $\mathcal{C} =\bigcup_i\widetilde{\mathcal{C}}_i$,
\begin{equation*}\label{gamn}
S_{r, \theta}^{s}(E) \leq \sum_{B \in \mathcal{C}} |B|^s \leq c_n (4^{n}+2\cdot4^{2n} d_n)(D+1)\frac{r^s}{\gamma} = a_n\ceil{\log_2(|E|/r) +1}\frac{r^s}{\gamma}
\end{equation*}
on setting $a_n = c_n (4^{n}+2\cdot4^{2n} d_n)$. 
\end{proof}
\end{lemma}

\section{Intermediate dimensions of Projections}\label{projsec}

Our main theorem in this section is that the intermediate dimension profiles $\lid^m E$ and $\uid^m E$ give the almost sure constant values of the lower and upper intermediate dimensions of orthogonal projections of $E$ onto $m$-dimensional subspaces. Thus, intuitively, we can regard $\lid^m E$ and $\uid^m E$ as the intermediate dimensions of $E$ when regarded from an $m$-dimensional viewpoint. Let $\gamma_{n, m}$ be the natural invariant measure on the Grassmannian $G(n, m)$ of $m$-dimensional subspaces of $\R^n$, see \cite[Section 3.9]{mat:book}.

\begin{theorem}\label{main}
Let $E \subset \R^n$ be bounded. Then, for  all $V \in G(n, m)$
\begin{equation}\label{mains}
\lid \pi_V E \leq \lid^{m} E \quad\mathrm{and}\quad \uid \pi_V E \leq \uid^{m} E
\end{equation}
for all    $\theta\in (0,1]$.  Moreover, for $\gamma_{n, m}$-almost all $V \in G(n, m)$,
\begin{equation}\label{mainas}
\lid \pi_V E = \lid^{m} E \quad\mathrm{and}\quad \uid \pi_V E = \uid^{m} E
\end{equation}
for  all $\theta\in (0,1]$.
\end{theorem}

To prove Theorem \ref{main} we begin with some technical lemmas relating the  kernel $\phi_{r, \theta}^{s,m}$ to the integral over $V$ of certain kernels defined on  $V \in G(n, m)$. We derive this from a standard estimate on integrals of the characteristic functions of slabs, which has been used in several results on projections, see for example \cite[Lemma 3.11]{mat:book} and  \cite{fal:2019}. The next lemma states this standard fact; we indicate the  proof for the lower bound which does not seem readily accessible.
For this we use the kernels  
\be\label{kerbd}
\phi_r^m(x) = \min\bigg\{1, \left(\frac{r}{|x|}\right)^m\bigg\}\qquad (x\in \R^n)
\ee
for $r>0$ and $m>0$ which were used in \cite{fal:2019} in connection with box-counting dimensions of projections.

\begin{lemma}\label{matext}
There exist constants $c_{n,m}, d_{n,m}> 0$ depending only on $n$ and $m$ such that for all $x \in \R^n$ and $r \in (0, 1)$,
$$
c_{n,m}\phi_r^m(x) \leq \int 1_{[0, r]}(|\pi_V x|) d\gamma_{n, m}(V) \leq d_{n,m}\phi_r^m(x).
$$
\begin{proof}
The right-hand inequality is given in \cite[Lemma 3.11]{mat:book}. The left-hand inequality is obvious when $|x| \leq r$, otherwise we may adapt the proof of  \cite[Lemma 3.11]{mat:book} by using the estimate 
\begin{align*}
\sigma^{n-1}\bigg(&\Big\{y \in S^{n-1} : \Big(\sum\limits_{i = m + 1}^{n}y_i^2\Big)^{1/2} \leq r\Big\}\bigg)\\
&\geq \ \alpha(n)^{-1}\mathcal{L}^n\left(\left\{y \in \R^n: |y_i| \leq 1/2 \textnormal{ for }i \leq m, |y_i| \leq r / n \textnormal{ for } i > m\right\}\right),
\end{align*}
where $\sigma^{n-1}$ denotes the normalised surface measure on $S^{n-1}$, \ $\alpha(n)$ is the volume of the unit ball in $\R^n$ and $\mathcal{L}^n$ is $n$-dimensional Lebesgue measure.
\end{proof}
\end{lemma}

It is convenient to introduce further kernels $\widetilde{\phi}_{r, \theta}^s$ on $m$-dimensional subspaces, where $0<r<1, \theta\in (0,1]$ and $0<s\leq m$
\be\label{kerhat}
\widetilde{\phi}_{r, \theta}^s(x)=
\begin{cases}
1 & |x| < r\\ 
\big(\frac{r}{|x|}\big)^s & r \leq |x| \leq r^\theta\\
0 & r^\theta < |x|
\end{cases} 
\qquad (x\in V),
\ee
where $V\in G(n,m)$ is some $m$-dimensional subspace.
The motivation for this is that  whilst $\widetilde{\phi}_{r, \theta}^s$ is of the same form as  ${\phi}_{r, \theta}^{s,m}(x)$  in the key region $|x| \leq r^\theta$  integrating $\widetilde{\phi}_{r, \theta}^s(\pi_V x)$ over $V\in G(n,m)$ gives a kernel comparable to $\phi_{r, \theta}^{s, m}(x)$. 
For brevity, we write $\simeq$ to mean that the ratio of the two sides is bounded away from $0$ and infinity by constants that are uniform in $x, r$ and $\theta$.
\begin{lemma}\label{kernellem}
For all $m \in \{1, \dots, n-1\}$ and $0 \leq s < m$ there exist constants $a, b> 0$, depending only on $n,m$ and $s$, such that for all $x \in \R^n$, $\theta \in (0, 1]$ and $0< r < \frac{1}{2}$,
$$
a_{n,m}\int \widetilde{\phi}_{r, \theta}^s(\pi_V x) d\gamma_{n, m}(V) \leq\phi_{r, \theta}^{s, m}(x) \leq b_{n,m}\int \widetilde{\phi}_{r, \theta}^s(\pi_V x) d\gamma_{n, m}(V).
$$
\end{lemma}

\begin{proof}
By direct integration, considering the cases $|x|\leq r,\ r <|x|\leq r^\theta$ and $r^\theta <|x|$ separately,
$$
\widetilde{\phi}_{r, \theta}^s(x) = sr^s\int\limits_{u = r}^{r^\theta} 1_{[0, u]}(|x|)u^{-(s+1)}du\ + \ r^{s(1-\theta)}1_{[0, r^\theta]}(|x|),
$$
and so using Fubini's theorem
\begin{align*}
\int \widetilde{\phi}_{r, \theta}^s(\pi_Vx) d\gamma_{n, m}(V) &= \int \bigg[sr^s\int\limits_{u = r}^{r^\theta} 1_{[0, u]}(|\pi_V x|)u^{-(s+1)}du \\
&\quad + r^{s(1-\theta)}1_{[0, r^\theta]}(|\pi_V x|)\bigg]d\gamma_{n, m}(V) \\
&= sr^s \int\limits_{u = r}^{r^\theta} u^{-(s+1)} \left[\int 1_{[0, u]}(|\pi_V x|)d\gamma_{n, m}(V)\right]\,du\\
&\quad + r^{s(1-\theta)}\int 1_{[0, r^\theta]}(|\pi_V x|)d\gamma_{n, m}(V).
\end{align*}

Using Lemma \ref{matext} and computing the integral using \eqref{kerbd} yields
\begin{align*}
&\int \widetilde{\phi}_{r, \theta}^s(\pi_Vx) d\gamma_{n, m}(V) 
  \simeq  sr^s \int\limits_{u = r}^{r^\theta} u^{-(s+1)}\phi_u^m(x)\,du + r^{s(1-\theta)}\phi_{r^\theta}^m(x)\\
 &= 
 \begin{cases}
      1 & (|x| < r) \\
       \frac{s}{m - s}\left(\left(\frac{r}{|x|}\right)^s - \left(\frac{r}{|x|}\right)^{m}\right) + \left(\frac{r}{|x|}\right)^s  &( r\leq |x| \leq r^\theta) \\
      \frac{s}{m-s} |x|^{-m}(r^{\theta(m-s)+s}- r^m)+|x|^{-m} r^{\theta(m-s) + s} & (r^\theta < |x|) 
   \end{cases}
  \\
&\simeq \phi_{r, \theta}^{s, m}(x)
\end{align*}
by comparing with \eqref{ker}, where the implied constants are uniform for $x\in \R^n$, $\theta\in(0,1]$ and $0<r<\frac{1}{2}$.
\end{proof}

Note that Lemma \ref{kernellem} is not quite valid when $s=m$ since a logarithmic term appears in the final integral. However we can avoid this  case in our application.

We require one further lemma which is a variant of Lemma \ref{capacitylb} for the modified kernels $\widetilde{\phi}_{r, \theta}^{s}$.
\begin{lemma}\label{kerswap}
Let $E\subset \R^n$ be compact, $\theta \in (0, 1]$, $0<r<1$ and $0 \leq s \leq n$. If there exists $\mu \in \mathcal{M}(E)$  and a Borel set $F \subset E$ such that
$$
\int \widetilde{\phi}_{r, \theta}^{s}(x - y) d\mu(y) \leq {\gamma}
$$
for  all $x \in F$, then
$$\mu(F) r^{s}{\gamma}^{-1}\leq S_{r, \theta}^s(E),  $$
where $S_{r, \theta}^s(E)$ is given by \eqref{sums}.
\begin{proof}
As in  Lemma \ref{capacitylb},
\begin{equation*}
{\gamma} \geq \int \widetilde{\phi}_{r, \theta}^{s}(x - y) d\mu(y) 
\geq \left(\frac{r}{\delta}\right)^s \mu(B(x, \delta))
\end{equation*}
for all $x \in F$ and $r \leq \delta \leq r^\theta$. Let $\{U_i\}_{i }$ be a cover of $F$ by sets with $r \leq |U_i| \leq r^\theta$. We may assume that for each $i$ there is some $x_i \in F\cap U_i$, so that $U_i \subset B(x_i, |U_i|)$. Hence
$$
 \mu(F) \leq \sum\limits_{i} \mu(U_i)\leq  \sum\limits_{i} \mu(B(x_i, |U_i|))\leq r^{-s}{\gamma} \sum\limits_{i}|U_i|^s,
$$
so taking infima over all such covers,
$$
S_{r, \theta}^{s}(E) \geq S_{r, \theta}^{s}(F) \geq  \mu(F) r^{s}{\gamma}^{-1}.
$$
\end{proof}
\end{lemma}

\noindent{\bf Proof of Theorem \ref{main}}
To prove Theorem \ref{main} it suffices to prove the \emph{a priori} weaker result where we first fix $\theta \in (0,1]$ and then establish the result for almost all $V$.  We can do this because the intermediate dimensions are continuous in   $\theta \in (0,1]$ and are therefore determined by their values on the rationals. 

Without loss of generality let $E \subset \R^n$ be compact and $m \in \{1, \dots, n-1\}$. When $\theta =1$ Theorem 5.1 reduces to the projection properties for box-counting dimensions, see \cite{fal:2019}. Hence we will assume that $\theta \in (0,1)$.

To obtain the upper bounds \eqref{mains} we will apply Lemma \ref{capacityub} to projections $\pi_V E$ of $E$ onto  
$V \in G(n, m)$. 
It is clear from the definition of  $\phi_{r, \theta}^{s, m}$  \eqref{ker} that, for all $0 \leq s \leq m,\  \theta \in (0,1)$ and $0<r<1$,  $\phi_{r, \theta}^{s, m}(x)$ is monotonically decreasing in $|x|$. Since orthogonal projection is contracting, that is  $|\pi_V(x) - \pi_V(y)| \leq |x-y|$, it follows  that 
$$\phi_{r, \theta}^{s, m}(\pi_V(x) - \pi_V(y))\geq \phi_{r, \theta}^{s, m}(x- y)  \qquad (x,y \in  \R^n).$$

By Lemma \ref{attaincap}, for each $0 \leq s \leq m$ there exists a measure $\mu \in \mathcal{M}(E)$ such that for all $x \in E$
\begin{align*}
\frac{1}{C_{r, \theta}^{s, m}(E)} &\leq \int \phi_{r, \theta}^{s, m}(x- y) d\mu(y) \\&\leq \int \phi_{r, \theta}^{s, m}(\pi_V(x) - \pi_V(y)) d\mu(y)\\
&\leq \int \phi_{r, \theta}^{s, m}(\pi_V(x) - w) d\mu_V(w),
\end{align*}
where $\mu_V \in \mathcal{M}(\pi_VE)$ denotes the image of $\mu$ under $\pi_V$ defined by $\int g(w)d\mu_V(w) = \int g(\pi_Vx)d\mu(x)$ for continuous $g$ and by extension. Then, for each $z = \pi_V(x) \in \pi_V E$, 
\begin{equation*}
\int \phi_{r, \theta}^{s, m}(z - w) d\mu_V(w) \geq \frac{1}{C_{r, \theta}^{s, m}(E)}.
\end{equation*}
Thus $\pi_V E\subset V$ supports a measure $\mu_V$ satisfying the condition of Lemma \ref{capacityub} (with $n$ replaced by $m$ and $V$ identified with $\R^m$ in the natural  way). Hence
$$ S_{r,\theta}^s(\pi_V E) \leq a_n\ceil{\log_2(|E|/r) + 1}r^{s}C_{r, \theta}^{s, m}(E)$$
for all $0 \leq s \leq m$ for sufficiently small $r$. Thus
$$
 \liminf\limits_{r \rightarrow 0} \frac{S_{r,\theta}^{s}(\pi_V E)}{-\log r} \leq -s + \liminf\limits_{r \rightarrow 0} \frac{C_{r,\theta}^{s, m}(E)}{-\log r},
$$
and the definitions  \eqref{altlid} and \eqref{lidp} imply that $\lid \pi_V E \leq \lid^{m} E $. The inequality for upper intermediate dimensions follows in the same way on taking upper limits.

To show that the opposite inequalities hold for almost all $V$ let $\theta \in (0,1)$ and $0 \leq s < m$. Let $(r_k)_{k \in \N}$ be a sequence tending to 0 such that $0 < r_k \leq 2^{-k}$ and
\be\label{supest}
\uplim\limits_{k \rightarrow \infty} \frac{\log C_{r_k, \theta}^{s, m}(E)}{-\log r_k} = \uplim\limits_{r \rightarrow 0} \frac{\log C_{r, \theta}^{s, m}(E)}{-\log r}.
\ee
Using Lemma \ref{attaincap}, for each  $k \in \N$, let  $\mu^k$ be an equilibrium measure on $E$ for the kernel  $\phi_{r_k, \theta}^{s, m}$   and let 
$$
\gamma_k := \frac{1}{C_{r_k, \theta}^{s, m}(E)} = \int \int \phi_{r_k, \theta}^{s, m}(x- y)d\mu^k(x)d\mu^k(y).
$$
With $\widetilde{\phi}_{r, \theta}^{s}$ as in \eqref{kerhat},  Lemma \ref{kernellem} gives 
\begin{align*}
&\int \int \int \widetilde{\phi}_{r_k, \theta}^{s}(\pi_Vx- \pi_Vy)d\gamma_{n, m}(V)d\mu^k(x)d\mu^k(y) \leq a^{-1}{\gamma_k}.
\end{align*}
Then for each  $\epsilon > 0$, 
\begin{equation*}
\int \int \int \gamma_k^{-1}r_k^{\epsilon}\widetilde{\phi}_{r_k, \theta}^{s}(\pi_Vx- \pi_Vy)d\gamma_{n, m}(V)d\mu^k(x)d\mu^k(y) \leq a^{-1}r_k^{\epsilon},
\end{equation*}
so summing and using Fubini's Theorem,
\begin{align*}
\int \sum\limits_{k=1}^{\infty}\left(\int \int \gamma_k^{-1}r_k^{\epsilon}\widetilde{\phi}_{r_k, \theta}^{s}(\pi_Vx- \pi_Vy)d\mu^k(x)d\mu^k(y)\right)d\gamma_{n, m}(V)\ \\ \leq\ a^{-1}\sum\limits_{k=1}^{\infty}r_k^{\epsilon} < \infty
\end{align*}
since $r_k^\epsilon \leq 2^{-k\epsilon}$. Hence, for $\gamma_{n,m}$-almost all $V$, there exists $M_V>0$ such that
\begin{align*}
\int \int \gamma_k^{-1}r_k^{\epsilon}\widetilde{\phi}_{r_k, \theta}^{s}(t- u)d\mu^k_V(t)d\mu^k_V(u) \leq M_V < \infty
\end{align*}
for all $k$, where $\mu_V^k \in \mathcal{M}(\pi_VE)$ is the image of the measure $\mu^k$ under $\pi_V$. Hence for such $V$,
\begin{align*}
\int \int \widetilde{\phi}_{r_k, \theta}^{s}(t- u)d\mu^k_V(t)d\mu^k_V(u) \leq M_V \gamma_k r_k^{-\epsilon}
\end{align*}
for all $k$. Thus, for each $k$ there exists a  set $F_k \subset  \pi_V E$ such that $\mu_V^k(F_k) \geq \frac{1}{2}$ and 
$$
\int \widetilde{\phi}_{r_k, \theta}^{s}(t- u)d\mu_V^k(t) \leq 2M_V \gamma_k r_k^{-\epsilon}
$$
for all $u \in F_k$. It  follows from Lemma \ref{kerswap} that
\begin{align*}
S_{r_k, \theta}^s(\pi_V E) \geq {\textstyle\frac{1}{2}} (2M_V\gamma_k)^{-1}r_k^{s+ \epsilon},
\end{align*}
and so
\begin{align*}
\uplim\limits_{k \rightarrow \infty}\frac{\log S_{r_k, \theta}^{s}(\pi_V E)}{-\log r_k} &\geq \uplim\limits_{k \rightarrow \infty}\frac{\log r_k^{s + \epsilon}(4M_V\gamma_k)^{-1}}{-\log r_k}\\
&= \uplim\limits_{k \rightarrow \infty}\frac{\log r_k^{s + \epsilon}C_{r_k, \theta}^{s, m}(E)}{-\log r_k}\\
&= -(s + \epsilon) + \uplim\limits_{k \rightarrow \infty} \frac{\log C_{r_k, \theta}^{s, m}(E)}{-\log r_k}.
\end{align*}
This is true for all $\epsilon >0$, so using \eqref{supest},
$$
\uplim\limits_{r\to 0}\frac{\log S_{r, \theta}^{s}(\pi_V E)}{-\log r} 
\geq
-s  + \uplim\limits_{r\to 0} \frac{\log C_{r, \theta}^{s, m}(E)}{-\log r}
$$
 for all $s \in [0,m)$. Since the expressions on both sides of this inequality are continuous for $s \in [0,m]$  by Lemmas \ref{declem} and \ref{caplem}, the inequality is valid for $s \in [0,m]$. The definitions \eqref{altuid} and \eqref{ludp} now imply that $\uid \pi_V E \geq \uid^{m} E $ for almost all $V$.

The argument for lower intermediate dimensions \eqref{mainas} is similar,  setting $r_k = 2^{-k}$ and noting that the limits may be taken  through a geometric sequence of $r$ tending to 0. 
\hfill $\square$

\section{Observations and applications}

One of the most natural questions concerning the intermediate dimensions is that of continuity at $\theta=0$.  In particular, continuity at 0 provides a  complete continuous interpolation between the Hausdorff and box-counting dimensions and it is therefore of interest to establish this for various classes of set.  For example, this was demonstrated in  \cite[Proposition 4.1]{fafrke:2018}  for Bedford-McMullen self-affine carpets, despite the absence of a precise formula for the intermediate dimensions.  It turns out that continuity at 0 for the intermediate dimensions of a set implies continuity at 0 for the intermediate dimensions of the projections almost surely.

\begin{cor} \label{cor1}
Let $E \subset \mathbb{R}^n$ be a bounded set such that $\lid E$ is continuous at $\theta=0$.  If $V \in G(n, m)$ is such that $\hd \pi_V E = \min\{m, \hd E\},$ then $\lid \pi_V E $ is continuous at  $\theta=0$.  In particular, $\lid \pi_V E$ is continuous at $\theta =0$ for    $\gamma_{n, m}$-almost all $V \in G(n, m)$. A similar result holds for the upper intermediate dimensions.
\end{cor}

\begin{proof}
If $m\leq  \hd E$, then the result is immediate and so we may assume that $m>  \hd E$.  Then, for $\theta \in (0,1)$, using \eqref{mains}, Lemma \ref{monoinm},  Theorem \ref{equivdim}, and the assumption that $\lid E$ is continuous at $\theta=0$, we get
\[
\hd E \leq \hd \pi_V E \leq  \lid \pi_V E \leq   \lid^m E \leq   \lid^n  E =  \lid   E \to \hd E
\]
as $\theta \to 0$, which proves continuity of  $\lid \pi_V E $ at $\theta =0$.   The final part of the result, concerning almost sure continuity at 0,  follows from the above result together with the Marstrand-Mattila projection theorems for Hausdorff dimensions.
\end{proof}

\begin{cor} \label{cor2}
Let $E \subset \mathbb{R}^2$ be a Bedford-McMullen carpet associated with a regular $a \times b$ grid for integers $b>a\geq 2$.  Then $\lid \pi_V E$ and $\uid \pi_V E$  are continuous at $\theta=0$ for $\gamma_{2, 1}$-almost all $V \in G(2, 1)$. In particular, if $\log a /\log b \notin \mathbb{Q}$, then $\lid \pi_V E$ and $\uid \pi_V E$  are continuous at $\theta=0$ for all $V \in G(2, 1)$. 
\end{cor}

\begin{proof}
The almost sure result  follows immediately from Corollary \ref{cor1} and   \cite[Proposition 4.1]{fafrke:2018}.  The upgrade from almost all to all follows by applying  \cite[Theorem 1.1]{fjs}, which proved there are \emph{no} exceptions to Marstrand's projection theorem for Bedford-McMullen carpets of `irrational type' apart from possibly the projections onto the coordinate axes. However, the coordinate projections  are both  self-similar sets and therefore the intermediate dimensions are automatically continuous at $0$.  
\end{proof}

The converse implication in Corollary \ref{cor1} does not necessarily hold, since continuity at 0 for all of the projections of $E$ does not guarantee continuity at 0 for $E$.  For example, let $E$ be a set in the plane with $\hd E = 1$ that satisfies $\lid E = 2$ for all $\theta \in (0,1]$ and place it inside a circle. The existence of such an $E$ follows easily from the following consequence  of  \cite[Proposition 2.4]{fafrke:2018}.  Our capacity approach yields a simple proof, which we include for completeness.  
\begin{cor} \label{cor3}
If $E \subset \R^n$ is bounded and satisfies $\lbd E = n$, then  $
\lid E = \uid E =  n$ for all $\theta \in (0,1]$.  Similarly,  if $\ubd E = n$, then  $
 \uid E =  n$ for all $\theta \in (0,1]$.
\end{cor}
\begin{proof}
Observe that
$$
\lowlim\limits_{r \rightarrow 0}\frac{\log C_{r, \theta}^{n, n}(E)}{-\log r} =     \lbd E = n
$$
and so by \eqref{lidp} and Theorem \ref{equivdim} it follows $  \uid E \geq \lid  E =  \lid^n E = \lbd E = n$.  The result concerning $\uid E$ alone follows similarly.
\end{proof}

The following counter-intuitive result follows by piecing together Corollaries \ref{cor1} and \ref{cor3}.  This gives a concrete application of the intermediate dimensions to a question concerning only the box and Hausdorff dimensions.

\begin{cor}\label{cor4}
Let $E \subset \mathbb{R}^n$ be a bounded set such that $\lid E$ is continuous at $\theta=0$.  Then
\[
\lbd \pi_V E = m
\]
for $\gamma_{n, m}$-almost all $V \in G(n, m)$ if and only if
\[
\hd E \geq m.
\]
A similar result holds for upper dimensions replacing $\lid E$ and $\lbd E$ with $\uid E$ and $\ubd E$, respectively.
\end{cor}
\begin{proof}
One direction is trivial, and holds without the continuity assumption, since, if $\hd E \geq m$, then 
\[
m \geq \lbd \pi_V E  \geq \hd \pi_V E  \geq m
\]
for $\gamma_{n, m}$-almost all $V \in G(n, m)$.  The other direction is where the interest lies.  Indeed, suppose  $\lbd \pi_V E = m$ for $\gamma_{n, m}$-almost all $V \in G(n, m)$ but 
$\hd E < m$.  Then Corollary \ref{cor3} implies that $\lid  \pi_V E = m$ for $\gamma_{n, m}$-almost all $V \in G(n, m)$ and all $\theta \in (0,1]$.  Applying the Marstrand-Mattila projection theorem for Hausdorff dimension, it follows that for  $\gamma_{n, m}$-almost all $V \in G(n, m)$    $\lid  \pi_V E $ is not continuous at $\theta = 0$, which contradicts Corollary \ref{cor1}.
\end{proof}

To motivate Corollary \ref{cor4} we give a couple of simple applications.  If $E \subset \mathbb{R}^2$ is a Bedford-McMullen carpet satisfying $\hd E < 1 \leq \bd E$, then
\[
\ubd \pi_V E<1 = \min \{ \bd E , 1\}
\]
for    $\gamma_{2, 1}$-almost all $V \in G(2, 1)$.  This surprising application seems difficult to derive directly, noting that there is very little known about the \emph{box} dimensions of projections of Bedford-McMullen carpets, aside from them being almost surely constant.  Another, more accessible, example is provided by the sequence sets $F_p = \{ n^{-p} : n \geq 1\}$ for fixed $p >0$.  It is well-known that $\bd F_p = 1/(1+p)$ and therefore
\[
\bd (F_p \times F_p) = 2/(1+p)
\]
which is at least 1 for $p \leq 1$ and approaches 2 as $p$ approaches 0.   Continuity at $\theta = 0$ for $\uid F_p$ was established in \cite[Proposition 3.1]{fafrke:2018} and it is straightforward  to extend this to  $\uid(F_p \times F_p)$.  Therefore, since $\hd(F_p \times F_p)=0<1$, we get
\[
\ubd \pi_V (F_p \times F_p) < 1
\]
for  $\gamma_{2, 1}$-almost all $V \in G(2, 1)$.  This is most striking when $p$ is very close to 0.  A direct calculation, which we omit, in fact reveals that for all $V \in G(2, 1)$ apart from the horizontal and vertical projections
\[
\ubd \pi_V( F_p \times F_p)  =  1- \left(\frac{p}{p+1}\right)^2,
\]
an entertaining formula which we would not  have come across if Corollary \ref{cor4} had not lead us to it, see also \cite[Proposition 5.1]{frahoyu:2018}.

\section*{Acknowledgement}
SAB thanks the Carnegie Trust for financially supporting this work.  JMF and KJF are grateful for the financial support of \emph{EPSRC Standard Grant} (EP/R015104/1).

\end{document}